\documentclass[11pt]{article}
\usepackage{amsmath, amssymb, amsfonts, amsthm}
\usepackage{enumitem, multicol}
\allowdisplaybreaks
\def \Q {{\mathbb{Q}}}
\def \N {{\mathbb{N}}}
\def \Z {{\mathbb{Z}}}
\def \C {{\mathbb{C}}}

\def \ep {{\epsilon}}
\def \al {{\alpha}}

\DeclareMathOperator{\lcm}{lcm}

\newtheorem*{theorem*}{Theorem}
\newtheorem{theorem}{Theorem}
\newtheorem{cor}[theorem]{Corollary}

\newtheorem{pro}[theorem]{Proposition}

\title{Irreducibility criterion for certain trinomials}
\author {Biswajit Koley,  A.Satyanarayana Reddy \\
Department of 
Mathematics, Shiv Nadar 
University, India-201314\\ (e-mail: 
bk140@snu.edu.in, satyanarayana.reddy@snu.edu.in).
  }
\date{}
\begin{document}
\maketitle
\begin{abstract}
 In this article we study the irreducibility of polynomials  of the form 
 $x^n+\ep_1 x^m+p^k\ep_2$, $p$ being a prime number. We will show that they are irreducible for $m=1$. We have also provided the cyclotomic factors and reducibility criterion for trinomials of the form $x^n+\ep_1x^m+\ep_2$, where $\ep_i\in \{\, -1,+1\,\}$. This corrects few of the existing results of W. Ljuggren's on $x^n+\ep_1x^m+\ep_2$.
\end{abstract}
{\bf{AMS(2010)}}:11R09, 12D05,12E05\\
\section*{Introduction}\label{sec:intro}
E.S.Selmer~\cite{Sel} studied the irreducibility of  trinomials of the form   
$x^n\pm x^m \pm 
1$ over $\Q$. He provided a complete solution for $m=1.$
Later Ljunggren~\cite{Lju} extended Selmer's result for all $m>1$ and  proved for quadrimonials as well. A version of his result for trinomials is the following.
\begin{theorem}(Ljunggren)\label{thm:Lju}
Let $f(x)=x^n+\epsilon_1 x^m+\epsilon_2$ 
where $\ep_j\in \{-1,+1\}$. Then $f(x)$  has at most one irreducible non-reciprocal factor and 
 a reciprocal factor of $f(x)$ if any,  is the product cyclotomic polynomials.
\end{theorem}
Ljunggren provided the possible cyclotomic factors for trinomials, but they seemed to be incorrect in certain cases. For example, according to Theorem~3 of Ljunggren~\cite{Lju}, the polynomials 
$x^{50}-x^4-1$ and $x^{50}+x^{22}-1$ are divisible by  $x^4+x^2+1$ but they are 
divisible by  $x^4-x^2+1.$ Similarly if $d$ is even, $d_1\equiv 5 \pmod 6, d_2\equiv 1 \pmod 6$ 
and $d_3$ is odd, then 
$x^{dd_1}+x^d-1, x^{dd_2}-x^{2d}-1$ and $x^{2^{d_3}d}-x^d+1$ are divisible by 
$x^{2d}+x^d+1$ but they are actually  divisible by $x^{2d}-x^d+1.$

 Based on the above examples, we revisited Ljunggren's work and corrected those errors. For similar studies and related work, reader can look into ~\cite{Sel},~\cite{Lju},~\cite{tver},~\cite{mills},~\cite{F:J}. 

Immediately the question appears about the reducibility of polynomials of the form $x^n+\ep_1x^m+\ep_2 p$, $p$ being a prime. If $p$ is an odd prime, then the polynomials are irreducible directly follows from Proposition $1$ of \cite{LP}. Recently, the authors\cite{BS} have shown that $x^n+\ep_1x^m+2\ep_2$ has exactly one irreducible non-reciprocal factor apart from its cyclotomic factors. The method used there doesn't apply to the polynomials $x^n+\ep_1x^m+p^k\ep_2$ with $k\ge 2$. With a different approach, we will prove that 
\begin{theorem}\label{mainthm}
Suppose $f(x)=x^n+\ep_1 x+\ep_2 p^k$ be a polynomial of degree $n\ge 2$ with $p$ being a prime number and $\ep_i\in \{\, -1, +1\,\}, k\ge 2$. Then $f(x)$ is irreducible.
\end{theorem} 
For arbitrary $m$ there are, indeed, polynomials which are reducible. For example, 
\begin{align}
& x^5-x^2+4=(x^2+x+2)(x^3-x^2-x+2); \notag\\
& x^5-x^4+9=(x^2-3x+3)(x^3+2x^2+3x+3).\notag
\end{align}
More generally, $$x^{3n}+\ep_1x^{2n}+4\ep_1=(x^n+2\ep_1)(x^{2n}-\ep_1x^n+2),$$
 for every $n\ge 1$. Although $f(x)$ is reducible for $m>1$, we will show that $f(x)$ can not have more than $k$ factors. More precisely, 
\begin{theorem}\label{gen}
Suppose $p$ is a prime and $f(x)=x^n+\ep_1x^m+\ep_2 p^k$ with $\ep_i\in\{\, -1, +1\,\}$ be a polynomial of degree $n$ and $k\ge 2$. Then $f(x)$ is a product of atmost $k$ distinct non-reciprocal irreducible polynomials. 
\end{theorem} 
The separability of such polynomials has also been considered there.  Throughout the paper, we will consider the reducibility over $\Q$ (and hence over $\Z$) only. If $n$ is a positive integer, we define $e(n)$ as the largest even part of $n$, i.e. $n=2^{a}n_1$ with $n_1$ odd implies $e(n)=2^a$. 

\section*{Factorization of $x^n+\ep_1x^m+\ep_2$}
Let $f(x)=x^n+\ep_1x^m+\ep_2$ be a polynomial of degree $n$ with $\ep_i\in \{\, -1, +1\,\}$. From Theorem \ref{thm:Lju} $f(x)$ has a cyclotomic factor whenever it is reducible. To determine the reducibility criterion of $f(x)$, it is,  therefore, sufficient to find the cyclotomic factors of  $f(x)$. Before we start, we recall a few basic properties of cyclotomic polynomials which will be useful later.

\begin{pro}\label{pro}
Suppose $n$ is a positive integer and $\Phi_n(x)$ be the $n$th cyclotomic polynomial. 
\begin{enumerate}[label=(\alph*)]
\item Let $p$ be a prime. Then 
\begin{equation*}
\Phi_{pn}(x)=\begin{cases}
\Phi_n(x^p) & \mbox{ if $p|n$}\\
\Phi_n(x^p)/\Phi_n(x) & \mbox{ if $p\nmid n$}.
\end{cases}
\end{equation*}
\item\label{impor} Define $D_n^m=\{\, d\in \N\mid \lcm(m,d)=mn\,\}$. Then $\Phi_n(x^m)=\prod\limits_{d\in D_n^m}\Phi_d(x)$.
\item If $p$ is a prime and $(p,n)=1$ then $\prod\limits_{d|p^{\gamma}n}\Phi_d(x)=\prod\limits_{i=0}^{\gamma}\prod\limits_{d|n}\Phi_{p^id}(x).$ In particular, $x^n+1=x^{2n}-1/x^n-1=\prod\limits_{d|2n,d\nmid n}\Phi_d(x)=\prod\limits_{d|n}\Phi_{2d}(x)$. 
\item If $n,m$ are positive integers, then 
\begin{equation*}
\left( \prod\limits_{d|n}\Phi_d(x), \prod\limits_{d|m}\Phi_d(x)\right)=\prod\limits_{d|(n,m)}\Phi_d(x).
\end{equation*}
\end{enumerate}
\end{pro}

Considering the elementary nature we omit the detailed proof. One can look  into Thangadurai \cite{RT} for the same.\\ 

The polynomial $f(x)=x^n+\ep_1x^m+\ep_2 $ is reducible if and only if $\ep_2x^nf(x^{-1})=x^n+\ep_1\ep_2x^{n-m}+\ep_2$ is reducible. Therefore, for a given $n$  it is sufficient to consider the reducibility of polynomials $f(x)=x^n+\ep_1x^m+\ep_2$ with $n\ge 2m$. Throughout the section, we will consider $m=2^a\cdot 3^b\cdot M, n-2m=2^p\cdot 3^q\cdot N$ as the prime factorization of $m$ and $n-2m$ respectively. 
\begin{theorem}\label{tri1}
Let $f(x)=x^n-x^m-1$ be reducible. Then $q>b, e(m)>e(n-2m)$ and $f(x)$ is divisible by $\Phi_6(x^{(n,m)})$.
\end{theorem}

\begin{proof}
 Since $f(x)$ is reducible, from Theorem~\ref{thm:Lju} $f(x)$ 
has a 
 reciprocal factor. Consequently, there exits an $\al\in \C,$ where $\al \ne \pm 
1,0$ such that both $\al$ and 
 $\frac{1}{\al}$ are roots of $f(x).$ That is 
 $$\al^n-\al^m-1=0=\al^n+\al^{n-m}-1.$$ 
In other words,  $\al$ is a root of the polynomial 
$x^{n-2m}+1.$ This evenually implies $f(x)$ is irreducible for $n=2m$. 
So we need to consider $n>2m$ for the remaining part.  Since  $\alpha$ satisfies $x^{n-2m}+1=0$ and $x^n-x^m-1=0$, it would satisfy 
$x^{2m}+x^m+1=0$. In particular, $\alpha$ is a root of  $g(x)=\gcd(x^{n-2m}+1, 
x^{2m}+x^m+1).$

From Proposition \ref{pro}, it can be seen that  $x^{2m}+x^m+1=\prod\limits_{d \in D_3^m} \Phi_d(x),$ 
where $D_3^m=\{d\in \N|\lcm(m,d)=3m\}.$ If we consider the prime factorizations of $m$ and $n-2m$, then we have $x^{n-2m}+1= \prod\limits_{i=0}^{q} \prod\limits_{d|N} 
\Phi_{2^{p+1}3^id}(x)$ and
\begin{equation*}
\Phi_3(x^m)=\prod\limits_{i=0}^a \prod\limits_{d|M} \Phi_{2^i 3^{b+1}d}(x).
\end{equation*}
Let $n$ be odd so that $n-2m$ odd or  equivalently $p=0.$ Hence
\begin{eqnarray*}
 g(x) &=& \left( \prod\limits_{i=0}^a \prod\limits_{d|M} \Phi_{2^i 
3^{b+1}d}(x), \prod\limits_{i=0}^{q} \prod\limits_{d|N} \Phi_{2.3^id}(x)
\right) = \left( \prod\limits_{d|M} \Phi_{2. 3^{b+1}d}(x), 
\prod\limits_{i=0}^{q} \prod\limits_{d|N} \Phi_{2.3^id}(x) \right)\\
&=& \begin{cases}
\prod\limits_{d|d_1} \phi_{2.3^{b+1}d}(x) \, , \mbox{ if $q > b$, 
where $d_1=(N,M)$}\\
1\, , \quad \qquad \quad\mbox{ otherwise.}
\end{cases}
\end{eqnarray*}
If $q \ge b+1$ then $n=2^{a+1}3^{b}M+3^{q}N=3^{b}u_3$ 
where $u_3$ odd and $3\nmid u_3$.\\
Also, $(n-2m,m)=(n,m)=3^{b}d_1$ gives $\prod\limits_{d|d_1} \phi_{2.3^{b+1}d}(x)
=\prod\limits_{d|d_1} \phi_{6.3^bd}(x)=\Phi_6(x^{(n,m)}).$

On the other hand, if $n$ is even then $n-2m$ is even and $p \ge 1$.  Then 
\begin{eqnarray*}
g(x) &=& \left(  \prod\limits_{i=0}^a \prod\limits_{d|M} \Phi_{2^i 
3^{b+1}d}(x), \prod\limits_{i=0}^{q} \prod\limits_{d|N} 
\Phi_{2^{p+1}.3^id}(x) \right)\\
&=& \begin{cases}
\prod\limits_{d|d_2} \Phi_{2^{p+1}.3^{b+1}d}(x)\, , \mbox{ if $a>p, q> b$, where $d_2=(N,M)$}\\
1\, , \qquad \qquad \qquad \mbox{ otherwise.}
\end{cases} \notag
\end{eqnarray*}
If $a\ge p+1, q\ge b+1$ then 
$n=2^{a+1}3^{b}M+2^{p}3^{q}N=2^{p} 3^{b}u_4$ where 
$(u_4,6)=1$.
\end{proof}

\begin{cor}
If $n=2^a3^b$ with $a+b>0$ then $x^n-x^m-1$ is irreducible for every $m<n$. 
\end{cor}
Since the proof for the remaining three families are almost same, instead of duplicating we state them without proof. The detailed proof can be carried out by using Proposition \ref{pro} and Theorem \ref{tri1}.
\begin{theorem}
Let $f(x)=x^n+\ep_1x^m-\ep_1$ be reducible with $\ep_1\in \{\, -1, +1\,\}$. Then $f(x)$ is divisible by $\Phi_6(x^{(n,m)})$ and the following holds:
\begin{multicols}{2}
\begin{enumerate}[label=(\alph*)]
\item $\ep_1=1, e(m)=e(n-2m), q>b$; 
\item $\ep_1=-1, e(m)<e(n-2m), q>b$.
\end{enumerate}
\end{multicols}
\end{theorem}

\begin{theorem}
Let $f(x)=x^n+x^m+1$ be reducible. Then $f(x)$ is divisible by $\Phi_3(x^{(n,m)})$ and either of the following holds necessarily 
\begin{multicols}{2}
\begin{enumerate}[label=(\alph*)]
\item $n=2m, M>1$; \item $n\ne 2m, q>b$.
\end{enumerate}
\end{multicols}
\end{theorem}

If we summarize all the results of this section, it fits perfectly within the below table.
\begin{description}
 \item[If $n=2m$] then $x^n\pm x^m-1$ are irreducible. And 
$x^n- x^m+1=\Phi_6(x^m)$, $x^n+x^m+1=\Phi_3(x^m)$  are reducible or 
irreducible according
to Proposition~\ref{pro}\ref{impor}.
\item[If $n\ne 2m$] then  \begin{enumerate}
\item If $m$ is odd then $x^n-x^m-1$ is irreducible.
\item If $m+n$ is odd then $x^n+x^m-1$ is irreducible.
\item If $n$ is odd then $x^n-x^m+1$ is irreducible.
\end{enumerate}
\end{description}
The following table summarizes the irreducibility of all polynomials for $n\ne 2m.$
Suppose $m=2^a\cdot 3^b\cdot M$, $n-2m=2^p\cdot 3^q\cdot N.$ And $F$ is one of the nontrivial 
reciprocal 
factor of $x^n\pm x^m\pm 1.$ \\
\vglue 2mm
\begin{tabular}{|l|l|l|l|l|l|}
\hline
$m$ & $n$ & $x^n-x^m-1$ & $x^n+x^m-1$ & $x^n-x^m+1$ & $x^n+x^m+1$ \\
\hline
even & odd & reducible if $q>b$ & irreducible & irreducible & reducible if 
$q>b$ 
\\
 & & $F=\Phi_6(x^{(n,m)})$  & & 
&$F=\Phi_3(x^{(n,m)})$ \\
 \hline
even & even & reducible   & reducible & 
reducible  &same as above   \\
&&if $q>b,a>p$& if $a=p$, $q>b$&if $p> a, q> b$&\\
 & & $F=\Phi_6(x^{(n,m)})$ 
 &   $F=\Phi_6(x^{(n,m)})$ &$F=\Phi_6(x^{(n,m)})$  & \\
 \hline
 odd & even & irreducible & irreducible &same as above  & same as above 
 \\
\hline
 odd & odd & irreducible & same as even-even  & irreducible &same as above
\\
 \hline
\end{tabular}\\

\section*{Factorization of $x^n+\ep_1 x^m+p^k\ep_2$}
Suppose $f(x)=x^n+\ep_1x^m+p^k\ep_2$ be a polynomial of degree $n$ with $\ep_i\in \{\, -1, +1\,\}, k\ge 2$. We will first prove the separability of such polynomials using discriminant. It is known that 

\begin{theorem}\cite{GD}\label{th0}
The discriminant of the trinomial $x^n+ax^m+b$ is 
\begin{equation*}
D=(-1)^{\binom{n}{2}}b^{m-1}\left[ n^{n/d}b^{n-m/d}-(-1)^{n/d}(n-m)^{n-m/d}m^{m/d}a^{n/d}\right]^d
\end{equation*}
where $d=(n,m)$.
\end{theorem}

\begin{theorem}\label{separable}
Let $p$ be a prime. The polynomial $f(x)=x^n+\ep_1x^m+p^k\ep_2$ is separable over $\Q$, $\ep_i\in \{\, -1, +1\,\}$. 
\end{theorem}
\begin{proof}
By Theorem \ref{th0}, the discriminant of $f(x)$ is 
\begin{equation}\label{eqd}
D_f=(-1)^{\binom{n}{2}}(p^k\ep_2)^{m-1}\left[ n^{n/d}(p^k\ep_2)^{n-m/d}-(-\ep_1)^{n/d}(n-m)^{n-m/d}m^{m/d}\right]^d
\end{equation}
with $d=(n,m)$. Since $f(x)$ is separable over $\Q$ if and only if $f(x^d)$ is separable, it is sufficient to consider $d=1$. $f(x)$ has multiple root if and only if $D_f=0$. Then from \eqref{eqd}, we have
\begin{equation*}
n^n(p^k\ep_2)^{n-m}=(-\ep_1)^n(n-m)^{n-m}m^m
\end{equation*}
which is not possible as $d=1$ and $p$ being a prime. 
\end{proof}

\begin{theorem}\label{th1}
Let $p$ be a prime and $f(x)=x^n+\ep_1x^m+\ep_2p^k$ be a polynomial of degree $n$ with $k\ge 2$. Then $f(x)$ has all its root on the region $|z|>1$.
\end{theorem}

\begin{proof}
Let $z_1$ be a root of $f(x)$ with $|z_1|\le 1$. Then $f(z_1)=0$ gives 
\begin{equation*}
p^k\ep_2=-(z_1^n+\ep_1z_1^m).
\end{equation*}
Taking modulus on both side gives $p^k=|z_1^n+\ep_1z_1^m|\le |z_1|^n+|z_1|^m\le 2$, which contradicts the fact that $p$ is a prime number and $k\ge 2$. Hence all the roots of $f(x)$ lies in the region $|z|>1$. 
\end{proof}
By using this theorem, we will prove Theorem \ref{gen}.
\begin{proof}[\unskip\nopunct]{\textbf{Proof of Theorem \ref{gen}:}}
Suppose $f(x)=\prod\limits_{i=1}^t f_i(x)$ be non-trivial factorization of $f(x)$, where each $f_i(x)$ is irreducible. Since $f(x)$ is a  monic polynomial,  we assume that each $f_i(x)$ is monic. From Theorem \ref{separable} $f(x)$ being separable, $f_i(x)\ne f_j(x)$ for $i\ne j$. By using Theorem \ref{th1}, from $|f(0)|=p^k=\prod\limits_{i=1}^t |f_i(0)|$, we have $|f_i(0)|\ge p$. In other words, $t\le k$ and consequently they are non-reciprocal. 
\end{proof}
Now we will prove the irreducibility of $x^n+\ep_1x+\ep_2p^k$ for every $k\ge 2$. 

\begin{proof}[\unskip\nopunct]{\textbf{Proof of Theorem \ref{mainthm}:}}
If $f(x)=x^2+\ep_1x+p^k\ep_2$ is reducible then all of its roots are integers only, by Rational Root theorem. But $u(u+\ep_1)=-\ep_2p^k$ is not possible for any integer $u$.

Suppose $n\ge 3$ and $f(x)$ is reducible. Let $f(x)=f_1(x)f_2(x)$ be a non-trivial factorization of $f(x)$ with $\deg(f_1)=s$. Without loss of generality, we assume that both $f_1(x), f_2(x)$ are monic polynomials. From Theorem \ref{th1} we have $|f_i(0)|\ge p$. Since $f_1(0)f_2(0)=p^k\ep_2$, let $|f_1(0)|=p^v$ and $|f_2(0)|=p^{k-v}$ for some $v\ge 1$. We consider the following two polynomials
\begin{equation*}
g(x)=x^sf_1(x^{-1})f_2(x)=\sum\limits_{i=0}^n b_ix^i, \text{say}
\end{equation*}
and $\tilde{g}(x)=\sum\limits_{i=0}^n b_{n-i}x^i$. The way the polynomials has been defined, we have $f_1(0)=b_n$ and $f_2(0)=b_0$. Let $b_0=p^v\ep_2'$ with $|\ep_2'|=1$. Since $g(x)\tilde{g}(x)=x^nf(x)f(x^{-1})$, comparing the coefficients of $x^n$, we get
\begin{equation*}
 \sum\limits_{i=1}^{n-1} b_i^2=p^{2k}-p^{2(k-v)}-p^{2v}+2.
\end{equation*}
Suppose there are $r$ number of non-zero $b_i$'s, say $0<j_r<j_{r-1}<\cdots<j_1<n$ such that $b_{j_l}\ne 0$. Then $g(x)=b_nx^n+b_{j_1}x^{j_1}+\text{middle terms}+b_{j_r}x^{j_r}+b_0$ and 
\begin{equation}\label{eqt1}
g(x)\tilde{g}(x)=p^k\ep_2x^{2n}+b_nb_{j_r}x^{2n-j_r}+b_0b_{j_1}x^{n+j_1}+\cdots+p^k\ep_2.
\end{equation}
Whereas 
\begin{equation}\label{eqt2}
f(x)x^nf(x^{-1})=p^k\ep_2x^{2n}+\ep_1 x^{2n-1}+p^k\ep_1\ep_2x^{n+1}+\cdots+p^k\ep_2.
\end{equation}
Since $n\ge 3$, the second largest term in \eqref{eqt2} is $x^{2n-1}$ and has coefficient $\ep_1$. The second largest term in \eqref{eqt1} is either $x^{2n-j_r}$ or $x^{n+j_1}$ or both. That is either $j_r=1$ or $j_1=n-1$ or $n=j_r+j_1=1+(n-1)$ respectively. In all these cases, the corresponding coefficient is divisible by $p$ which is impossible. Therefore $f(x)$ has to be irreducible.
\end{proof}

\end{document}